\documentclass[a4paper]{amsart}
\usepackage{amssymb} 
\usepackage{amsmath} 
\usepackage{comment}
\usepackage{color}
\usepackage{array, booktabs, comment, mathrsfs, enumerate, multirow}
\usepackage[neverdecrease]{paralist}
\usepackage{tikz-cd}
\usepackage{todonotes}
\usepackage{microtype}
\usepackage{url}
\usepackage[colorlinks]{hyperref}
\hypersetup{citecolor=blue}

\DeclareMathOperator{\Art}{Art}
\DeclareMathOperator{\Aut}{Aut}

\DeclareMathOperator{\N}{N}
\DeclareMathOperator{\Cl}{Cl}

\DeclareMathOperator{\Hom}{Hom}
\DeclareMathOperator{\GL}{GL}
\DeclareMathOperator{\Jac}{Jac}

\DeclareMathOperator{\Pic}{Pic}

\DeclareMathOperator{\Gal}{Gal}

\DeclareMathOperator{\End}{End}

\DeclareMathOperator{\GSp}{GSp}
\DeclareMathOperator{\SO}{SO}

\DeclareMathOperator{\GSpin}{GSpin}

\DeclareMathOperator{\Stab}{Stab}

\DeclareMathOperator{\Ind}{Ind}

\DeclareMathOperator{\Tr}{Tr}

\DeclareMathOperator{\rec}{rec}
\DeclareMathOperator{\WD}{WD}

\newcommand{\A}{{\mathbf A}}
\newcommand{\Q}{{\mathbf Q}}
\newcommand{\Z}{{\mathbf Z}}
\newcommand{\C}{{\mathbf C}}
\newcommand{\R}{{\mathbf R}}
\newcommand{\F}{{\mathbf F}}

\newcommand{\CO}{\mathcal{O}}

\newcommand{\gc}{{\mathfrak{c}}}

\newcommand{\gM}{\mathfrak{M}}
\newcommand{\gm}{\mathfrak{m}}

\newcommand{\gP}{\mathfrak{P}}
\newcommand{\gp}{\mathfrak{p}}
\newcommand{\gq}{\mathfrak{q}}
\newcommand{\gQ}{\mathfrak{Q}}
\newcommand{\gN}{\mathfrak{N}}

\newcommand{\Qbar}{\overline{\Q}}

\newcommand{\uk}{\underline{k}}

\newcommand{\Frob}{\mathrm{Frob}}
\newcommand{\T}{\mathbf{T}}
\newcommand{\old}{\rm old}
\newcommand{\new}{\rm new}

\numberwithin{equation}{section}

\theoremstyle{plain}
\newtheorem{thm}[equation]{Theorem}
\newtheorem{lem}[equation]{Lemma}
\newtheorem{prop}[equation]{Proposition}

\newtheorem{cor}[equation]{Corollary}
\newtheorem{conj}[equation]{Conjecture}
\newtheorem{defn}[equation]{Definition}
\newtheorem{rem}[equation]{Remark}

\theoremstyle{definition}

\title[Base change action and Hecke orbits]{On the compatibility between base change and Hecke orbits of Hilbert newforms}

\begin{document}

\author{Lassina Demb\'el\'e}
\address{Department of Mathematics, King's College London, Strand WC2R 2LS, London, United Kingdom}
\email{lassina.dembele@gmail.com}
\thanks{The author was supported by EPSRC Grant EP/J002658/1 and a Visiting Scholar grant from the Max-Planck Institute for Mathematics}


\maketitle

\begin{abstract} Let $F/E$ be a Galois extension of totally real number fields, with Galois group $\Gal(F/E)$. 
Let $\gN$ be an integral ideal which is $\Gal(F/E)$-invariant, and $k \ge 2$ an integer.
In this note, we study the action of $\Gal(F/E)$ on the Hecke orbits of Hilbert newforms of level $\gN$
and weight $k$. We also discuss the geometric counterpart to this action, which is closely related to
the notion of abelian varieties potentially of $\GL_2$-type. The two actions have some consequences in 
relation with Langlands Functoriality. 

We conclude with an example over the maximal totally real subfield $F = \Q(\zeta_{32})^+$ of 
the cyclotomic field of 32nd root of unity. Let $D$ be the quaternion algebra over $F$ ramified exactly at the unique 
prime above $2$ and $7$ real places, and $X_0^D(1)$ the Shimura curve attached to $D$. Among other things, 
our example shows that the field of $2$-torsion of the Jacobian of the curve $X_0^D(1)$ (and its Atkin-Lehner quotient)
is the unique Galois extension $N/\Q$ unramified outside $2$, with Galois group the Frobenius group 
$F_{17} = \Z/17\Z \rtimes (\Z/17\Z)^\times$. This completes Noam Elkies' answer~\cite{elk15} to a question posed by 
Jeremy Rouse on \verb|mathoverflow.net|. 
\end{abstract}

\section*{\bf Introduction}
Let $F$ be a totally real number field of degree $d$, with ring of integers $\CO_F$. 
Let $\gN$ be an integral ideal of $F$ and $k\ge 2$ an integer. Let $f$ be a Hilbert newform
of level $\gN$, and (parallel) weight $k$; and $L_f$ the field of coefficients of $f$. 
We recall that $L_f = \Q(a_\gm(f): \gm \subseteq \CO_F)$, where $a_\gm(f)$ is the
Hecke eigenvalue of $f$ at the Hecke operator $T_\gm$ for the integral ideal $\gm$. 
By \cite[Proposition 2.8]{shimura78}, $L_f$ is either totally real or CM. 

Let $\tau:\,L_f \hookrightarrow \Qbar$ be a complex embedding. Then, by the Strong Multiplicity One 
Theorem~\cite{miy71} and \cite[Proposition 2.6]{shimura78}, there is a newform $f^\tau$ of level $\gN$
and weight $k$ determined by its Hecke eigenvalues
$$a_\gp(f^\tau) := \tau(a_\gp(f)),\,\,\text{for all primes}\,\, \gp.$$
The {\it Hecke orbit} of the form $f$ is defined as the finite set
$$[f] := \{f^\tau:\, \tau \in \Hom(L_f, \Qbar)\}.$$
Let us further assume that $F$ is Galois over some subfield $E$, and write $G:= \Gal(F/E)$.  
Let us also assume that $\gN = \gN^\sigma$ for all $\sigma \in G$. Then, similarly, for all $\sigma \in G$,
there exists a newform ${}^\sigma\!f$ of level $\gN$ and weight $k$ determined by its Hecke eigenvalues
 $$a_\gp({}^\sigma\!f) := a_{\sigma(\gp)}(f),\,\,\text{for all primes}\,\, \gp.$$
The form $f^\tau$ is often called an {\it exterior twist} while ${}^\sigma\!f$ is known as an {\it inner twist}.

In the literature, exterior twists and inner twists have been studied quite extensively, but separately, to the best of our knowledge. 
However, there is clearly an action of $G$ on Hecke orbits of newforms. Indeed, let $f$ be a newform of level $\gN$ and weight $k$,
$\sigma \in G$ and $\tau \in \Hom(L_f, \Qbar)$. Then, $L_{{}^\sigma\!f} = L_f$, and for $\gp$ prime, we have 
\begin{align*}
a_{\gp}({}^\sigma\!(f^\tau)) = a_{\sigma(\gp)}(f^\tau) = \tau(a_{\sigma(\gp)}(f))= \tau(a_{\gp}({}^\sigma\!f))=a_{\gp}(({}^\sigma\!f)^\tau).
\end{align*}
So that we have a well-defined action of $G$ on the set of Hecke orbits of newforms of level $\gN$ and weight $k$ given by 
$\sigma \cdot [f] = [{}^\sigma\!f]$. Therefore, understanding the action of $G$ on those orbits seems a rather natural question. 

In this short note, we describe the compatibility between base change and Hecke orbits of
Hilbert newforms when the base field is Galois. Our results are essentially a generalisation of those in~\cite{cd17}
to non-solvable extensions under the assumption that the Base Change Conjecture is true. Although the results are somewhat 
straightforward, and are probably known to most experts, they seem to have rather strong implications relating to Langlands 
Functoriality. That was illustrated in~\cite{cd17} where the functorial connection between the
Eichler-Shimura conjecture and the Gross-Langlands conjecture on the modularity of abelian varieties was discussed. 
We conclude the note with an example pertaining to this connection. Let $F$ be the maximal totally real subfield of $\Q(\zeta_{32})$, 
and $D$ the quaternion algebra over $F$ ramified at the unique prime above $2$ and $7$ real places. Let $X_0^D(1)$ be the
Shimura curve attached to $D$. Our example shows that the field of $2$-torsion of the Jacobian of the curve $X_0^D(1)$ 
(and its Atkin-Lehner quotient) is the unique Galois extension $N/\Q$ unramified outside $2$, with Galois group the
Frobenius group $F_{17} = \Z/17\Z \rtimes (\Z/17\Z)^\times$. This completes Noam Elkies' answer~\cite{elk15} to a question 
posed by Jeremy Rouse on \verb|mathoverflow.net|.  

The outline of the paper is as follows. In Sections~\ref{sec:hmf} and~\ref{sec:bc}, we recall the necessary background on Hilbert 
modular forms, and base change. In Section~\ref{sec:bc-and-hecke-comp}, we discuss the compatibility of base change with
Hecke orbits. In Section~\ref{sec:potential-gl2-av}, we introduce the notion of abelian varieties potentially of $\GL_2$-type. We then
discuss their Galois descent properties, and its implications for Langlands Functoriality. Finally, in Section~\ref{sec:gross-ex}, 
we conclude with an example.

\subsection*{Acknowledgements} We would like to thank Fred Diamond, Noam D. Elkies, Toby Gee, Xavier Guitart, and John Voight 
for helpful email exchanges and discussions. In particular, we are grateful to David P. Roberts for pointing out the 
\verb|mathoverflow.net| discussion related to the example in Section~\ref{sec:gross-ex}. Finally, a special thanks to
Benedict H. Gross who encouraged us to write up this note.

\section{\bf Hilbert modular forms}\label{sec:hmf}

In this section, we summarise the results we need on Hilbert modular forms,
their associated automorphic representations and Galois representations. 
We refer to~\cite{car86a, hid88, shimura78, tay89} for further details. 

Le $F$ be a totally real field, and $J_F$ the set of real embeddings of $F$.  Let $\mathfrak{H}_F:=\mathfrak{H}^{J_F}$, where 
$\mathfrak{H}=\{z\in\C|\mathrm{im}(z)>0\}$ is the Poincar\'e upper halfplane. Fix an integer $k \ge 2$. 

\subsection{Hilbert modular forms} We let $G = \mathrm{Res}_{F/\Q}(\GL_2)$, this is
the algebraic group obtained by restriction of scalars of $\GL_2$ from $F$ to $\Q$. By definition of $G/\Q$, we
have
$$G(\A_\Q) = G_f \times G_\infty = \GL_2(\A_F) = \GL_2(\A_F^f) \times \GL_2(\R)^{J_F}.$$ 
Let $G_\infty^+$
be the connected component of the identity element. Then, $G_\infty^+$ acts on $\mathfrak{H}_F$ 
component-wise by Mobius transforms. This action extends uniquely to $G_\infty$ component-wise,
such that on the $v$-th factor, the matrix $\tiny\begin{pmatrix}-1 &0 \\ 0 &1 \end{pmatrix}$ acts by
($z_v \mapsto -\overline{z}_v$). 

For every matrix $\gamma = \begin{pmatrix}a&b\\ c&d \end{pmatrix} \in \GL_2(F)$, and $z \in \C$, we define 
\begin{align*}
(cz + d)^{-k} &:= \prod_{v \in J_F} (c_v z_v + d_v)^{-k};\\
\det(\gamma)^{k - 1 } &:= \prod_{v \in J_F} \det(\gamma_v)^{k - 1}.
\end{align*} 

We let $U = \prod_{\gp} U_\gp$ be a compact open subgroup of $G(\widehat{\Z}) = \GL_2(\widehat{\CO}_F) \subset G_f$,
and consider the space of functions $f:\, G_f/U \times \mathfrak{H}_F \to \C$. 
There is a (left) action of $G_\Q$ on this space, given by
$$(f|_{k}\gamma)(x, z) := \det(\gamma)^{k - 1}  (cz + d)^{-k} f(\gamma x, \gamma z),\,x \in G_f,\,z \in \mathfrak{H}_F,$$
The space of {\it Hilbert modular forms}
$M_{k}(U)$ of weight $k$ and level $U$ is the set of functions $f:\, G_f/U \times \mathfrak{H}_F \to \C$ such that
\begin{enumerate}
\item $f|_{k}\gamma = f$ for all $\gamma \in G_\Q$; and 
\item for all $x \in G_f$, the map $\left(f_x:\,\mathfrak{H}_F \to \C,\, z \mapsto f(x, z)\right)$ is holomorphic.
\end{enumerate}
By  Condition (1) (and the Koecher principle~\cite[sec.~2]{shimura78}), $f_x$ admits a $q$-expansion of the form 
$$f_x(z) = a_0(f_x) + \sum_{\mu \in F_{+}^{\times}} a_\mu e^{2\pi i \Tr_{F/\Q}(\mu z)},$$
where $$\Tr_{F/\Q}(\mu z) := \sum_{v \in J_F}\mu_v z_v.$$ 
We say that $f$ is a {\it cusp form} if, in addition to (1) and (2), we have 
\begin{enumerate}
\setcounter{enumi}{2}
\item $f_x|_{k}\gamma$ has no constant term, i.e. $a_0(f_x|_{k}\gamma) = 0$, for all $x \in G_f$ and $\gamma \in G_\Q$. 
\end{enumerate}
We denote the space of cusp forms by $S_{k}(U)$. 

Let $\gN$ be an integral ideal, we defined the compact opens:
\begin{align*}
U_0(\gN) &:= \left\{\begin{pmatrix} a &b \\ c& d \end{pmatrix} \in \GL_2(\widehat{\CO}_F):\,\, c = 0 \bmod \gN \right\};\\
U_1(\gN) &:= \left\{\begin{pmatrix} a &b \\ c& d \end{pmatrix} \in \GL_2(\widehat{\CO}_F):\,\, c = 0 \bmod \gN,\,\,d = 1 \bmod \gN\right\}. 
\end{align*}
For $U= U_0(\gN)$ or $U_1(\gN)$, the strong approximation theorem implies that 
$$G_\Q\backslash G_f/U \simeq \A_F/F^\times \R_{+}^{J_F} \widehat{\CO}_F^\times
\simeq \Cl^{+}(F),$$
where $\Cl^{+}(F)$ is the narrow class group of $F$. Let $\mathfrak{d}$ be the different ideal of $F$, and $\gc_i$, $i = 1, \ldots, h^{+}$, a complete set of 
representatives for the classes in $\Cl^{+}(F)$. For each $i$, let $x_i \in G_f$ be such that the id\`ele $\det (x_i)$ 
generates the ideal $\gc_i$, and define 
\begin{align*}
\Gamma_0(\gc_i, \gN) &:=\left\{\begin{pmatrix} a &b \\ c& d \end{pmatrix} \in \begin{pmatrix}\CO_F & \gc_i^{-1}\mathfrak{d}^{-1}\\
\gc_i\mathfrak{d}\gN & \CO_F \end{pmatrix}: ad - bc \in \CO_F^{\times + }\right\};\\
\Gamma_1(\gc_i, \gN) &:=\left\{\begin{pmatrix} a &b \\ c& d \end{pmatrix} \in \Gamma_0(\gc_i, \gN): d = 1 \bmod \gN \right\}.
\end{align*}
Let $\Gamma_i = \Gamma_0(\gc_i, \gN)$ or $\Gamma_1(\gc_i, \gN)$ according as to $U= U_0(\gN)$ or $U_1(\gN)$.
Then, we have a bijection
\begin{align*}
S_{k}(U) & \to \bigoplus_{i = 1} ^{h^+} S_{k}(\Gamma_i)\\
f& \mapsto (f_i, \ldots, f_{h^+}),
\end{align*}
where $f_i = f_{x_i}$, and $S_{k}(\Gamma_i)$ is the space of classical Hilbert cusp forms
of level $\Gamma_i$ and weight $k$ (see~\cite[sec.~2]{shimura78}). 
Let $\gm \subseteq \CO_F$ be an ideal; then there is a unique $1\le i \le h^+$ and $\mu \in F_{+}^{\times}$ such that
$\gm = \mu \gc_i^{-1}$. We set 
$$a_\gm(f) = a_\mu(f_i).$$ 
Since $f_i$ is invariant under $\CO_F^{+ \times}$, which acts as $z \mapsto \epsilon z$, we 
see that 
$$a_{\epsilon \mu}(f_i) = a_\mu(f_i).$$ 
So the quantity $a_\gm(f)$ is well-defined. We call it the {\it Fourier coefficient} of $f$ at $\gm$. 
 
\subsection{Petersson inner product} We have the measure $d\mu(z) = \prod_{v \in J_F}y_v^{-2} d x_v d y_v$.
We define the Petersson inner product on $S_{k}(\Gamma_i)$ by 
$$\langle f, g\rangle_{\Gamma_i} = \frac{1}{\mu(\Gamma_i\backslash \mathfrak{H}_F)}\int_{\Gamma_i\backslash \mathfrak{H}_F}\overline{f(z)}g(z)y^k d \mu(z).$$  
The Petersson inner product on $S_{k}(U)$ is given by
$$\langle f, g\rangle = \sum_{i=1}^{h^+} \langle f_i, g_i\rangle_{\Gamma_i}.$$

\subsection{Hecke operators}

For $x \in G_f$, we define the Hecke operator
\begin{eqnarray*}
[UxU]:\, S_{k}(U)&\to&S_{k}(U)\\
f&\mapsto&\sum f|_{k}x_i,
\end{eqnarray*} where $[UxU]=\coprod_{i}x_i U$. For every prime $\gp$ (resp. $\gp \nmid \gN$), we define
the Hecke operators
$$T_{\gp}:=\left[U\begin{pmatrix}\varpi_\gp&0\\0 &1\end{pmatrix} U\right]\,\,\mbox{\rm and}\,\,
S_{\gp}:=\left[U\begin{pmatrix}\varpi_{\gp}&0\\0 &\varpi_{\gp}\end{pmatrix} U\right],$$ 
where $\varpi_\gp$ is a uniformizer at $\gp$ (resp. $\gp\nmid\gN$). We define the Hecke algebra 
$\T_{k}(U)$ to be the $\Z$-subalgebra of $\mathrm{End}_{\C}(S_{k}(U))$ generated by the operators
$T_{\gp}$ for all primes $\gp$, and $S_{\gp}$ for all primes $\gp \nmid \gN$. We will simply write $\T$ if there is no confusion. 

\subsection{Cusp forms with characters}
Let $\chi:\, F^\times \backslash \A_F^\times \to \C^\times$ be a Hecke character of modulus $\gN$ and infinite type $(-n_0,\ldots,-n_0)$. 
We define the space of Hilbert cups forms of level $\gN$, weight $k$ and character $\chi$ by
$$S_{k}(\gN, \chi) := \left\{f \in S_{\uk}(U_1(\gN)) :\,S_\gp f = \chi(\gp) f \right\}.$$
We will simply write $S_{k}(\gN)$ when $\chi := \mathbf{1}$ is the trivial character. 
Then, we have the following decomposition (see~\cite[sec.~2]{shimura78}), which is compatible with the Hecke action:
$$S_{k}(U_1(\gN)) = \bigoplus_{\chi} S_{k}(\gN, \chi),$$
where $\chi$ runs over all Hecke character of modulus $\gN$ and infinite type $(-n_0,\ldots,-n_0)$. We will denote by $\T_{k}(\gN, \chi)$
the Hecke algebra acting on $S_{k}(\gN, \chi)$. Again, we will simply write $\T$ if there is no confusion.

\subsection{Eigenforms}
We say that $f\in S_{k}(\gN, \chi)$ is an {\it eigenform} if it is an eigenvector for all the operators $T_\gm, S_\gm \in \T_{k}(\gN, \chi)$.
In addition, we say that $f$ is {\it normalised} if $a_{(1)}(f) = 1$. In this case, there is a ring homomorphism 
$\lambda_f:\,\T_{k}(\gN, \chi) \to \C$ such that $T_{\gm} f = \lambda_f(T_{\gm}) f$ for all $\gm \subseteq \CO_F$.
By Shimura~\cite[Proposition~2.8]{shimura78} we have $\lambda_f(T_{\gm}) = a_{\gm}(f)$, and the field
$L_f = \Q(a_\gm(f): \gm \subseteq \CO_F)$ is a number field. It is either totally real or ${\rm CM}$. 

\subsection{Newforms} Let $\gM \mid \gN$ be divisible by the conductor of $\chi$, and $\gQ \mid \gN \gM^{-1}$. 
Then, by the Multiplicity One Theorem, there is a well-defined map 
\begin{align*}
\iota_{\gQ}:\,S_{k}(\gM, \chi) &\to S_{k}(\gN, \chi)\\
g &\mapsto g|\gQ,
\end{align*}
determined by setting $a_{\gm}(g|\gQ) = a_{\gm \gQ^{-1}}(g)$. The {\it old subspace} of level $\gN$, weight $k$ and character
$\chi$ is defined by
$$S_{k}(\gN, \chi)^{\old} = \sum_{\gM \mid \gN \atop \gQ \mid \gN \gM^{-1}} \mathrm{im}(\iota_\gQ).$$
This is stable under the action of $\T_{k}(\gN, \chi)$. We define the {\it new subspace} $S_{k}(\gN, \chi)^{\new}$ of level $\gN$, weight $k$ and 
character $\chi$ to be the orthogonal complement of $S_{k}(\gN, \chi)^{\old}$ under the Petersson inner product. So, we have
$$S_{k}(\gN, \chi) = S_{k}(\gN, \chi)^{\old} \oplus S_{k}(\gN, \chi)^{\new}.$$ 
We say that $f \in S_{k}(\gN, \chi)$ is a {\it newform} if it is a normalised eigenform which belongs to $S_{k}(\gN, \chi)^{\new}$. 
By~\cite[Theorem 5.2]{hid88}, we have a perfect pairing
\begin{align*}
\T_{k}(\gN, \chi) \times S_{k}(\gN, \chi) \to \C,\,\, (T, f) \mapsto a_{(1)}(f|T).
\end{align*}
Let $f \in S_{k}(\gN, \chi)^{\new}$ be a newform. Then, we have a $\Q$-algebra homomorphism
\begin{align*}
\lambda_f:\,\T_{k}(\gN, \chi)^{\new}& \to \Qbar\\
T_{\gm} &\mapsto a_\gm(f). 
\end{align*}
By considering the $\Qbar$-structure  $S_{k}(\gN, \chi; \Qbar)$ on $S_{k}(\gN, \chi)$, we
have an isomorphism of $\Qbar$-vector spaces
\begin{align*}
S_{k}(\gN, \chi;\Qbar)^{\new}&\to\mathrm{Hom}_{\Qbar}(\T_{k}(\gN, \chi;\Qbar)^{\new},\Qbar)\\
f&\mapsto \lambda_f.
\end{align*}
We also have a natural bijection
\begin{align*}
\mathrm{Hom}_{\overline{\Q}\mbox{\rm\tiny -alg}}(\T_{k}(\gN,\chi; \Qbar)^{\new},\overline{\Q})
&\stackrel{\sim}{\to} \mathrm{Spec}(\T_{k}(\gN,\chi; \Qbar)^{\new})\\
\lambda &\mapsto \ker(\lambda).
\end{align*}

\subsection{Hecke orbits}
Let $f \in S_{k}(\gN, \chi)^{\new}$ be a newform. Let $\tau:\,L_f \hookrightarrow \Qbar$ be a complex embedding. 
Then, by the Strong Multiplicity One Theorem~\cite{miy71} and \cite[Proposition 2.6]{shimura78}, there is a newform 
$f^\tau$ of level $\gN$, weight $k$ and character $\chi^\tau := \tau \circ \chi$ determined by its Hecke eigenvalues
$$a_\gp(f^\tau) := \tau(a_\gp(f)),\,\,\text{for all primes}\,\, \gp.$$
The {\it Hecke orbit} of the form $f$ is defined as the finite set
$$[f] := \{f^\tau:\, \tau \in \Hom(L_f, \Qbar)\}.$$

\subsection{Automorphic representations attached to Hilbert newforms}
Let $f\in S_{k}(\gN, \chi)^{\new}$ be a newform. Then there is an automorphic representation $\pi_f$ of $\GL_2(\A_F)$ 
attached to $f$, which by~\cite{fla79} admits a factorisation as a restricted tension product 
$$\pi_f=\bigotimes_v\pi_{f,v},$$
where $\pi_{f, v}$ is an admissible representation of $\GL_2(F_v)$ for all places $v$ of $F$.

\subsection{Galois representations attached to Hilbert newforms}

We let $G_F := \Gal(\Qbar/F)$ be the absolute Galois group of $F$.
For each prime $\gp$, we also let $D_\gp$  and $W_{F_\gp}$ be the decomposition and the Weil group at $\gp$,
respectively. 

Let $f \in S_{k}(\gN, \chi)^{\new }$ be a newform, and $\pi_f$ the associated automorphic representation.
The following theorem is the result of the work of many people~\cite{car86a, tay89, bggt14b}.
\begin{thm}\label{thm:galois-reps-for-hmf} Let $\ell$ be a rational prime, and $\lambda$ a prime of 
$L_f$ above $\ell$. Then, there exists a Galois representation 
$$\rho_{f,\lambda}: \Gal(\Qbar/F) \to \GL_2(\overline{L}_{f,\lambda}),$$ such that 
\begin{enumerate}
\item[(1)] For all $\gp\nmid \gN\ell$, the characteristic polynomial of $\Frob_\gp$ is determined by
$$\mathrm{Tr}(\rho_{f,\lambda}(\Frob_\gp)) = a_{\gp}(f)\,\,\text{and}\,\, \det(\rho_{f,\lambda}(\Frob_\gp)) = \chi(\gp)\N\!\gp^{k-1}.$$
\item[(2)] More generally, we have $\WD(\rho_{f,\lambda}|_{D_\gp})^{\rm F-ss} \simeq \rec_{F_\gp}(\pi_{f, \gp}\otimes |\det|^{-1/2})$. 
\item[(3)] For $\gp \mid \ell$, $\rho_{f,\lambda}|_{D_\gp}$ is de Rham with Hodge-Tate weights 
$$\mathrm{HT}_\tau(\rho_{f,\lambda}) = \{k, k - 1\},$$
for each embedding $\tau: F \hookrightarrow \overline{L}_{f, \lambda} \simeq \C$ lying over the place at $\gp$.
If $\pi_{f, \gp}$ is unramified then $\rho_{f, \lambda}|_{D_\gp}$ is crystalline. 
\item[(4)] If $c_v$ is a complex conjugation, then $\det(\rho_{f,\lambda}(c_v)) = -1$. 
\end{enumerate}
\end{thm}

\section{\bf Base change}\label{sec:bc}
In this section, we recall the main conjecture regarding non-solvable base change for $\GL_2$,
and refer to~\cite{lang80} for more details. 

Let $F$ be as before and assume that $E$ is a subfield of $F$ such that $F/E$ is Galois. 
Assume that $\gN$ is stable under the action of $\Gal(F/E)$.  

Since $U$ is Galois stable, the group $\Gal(E/F)$ acts on $G_f/U \times \mathfrak{H}_F$ by
\begin{align*}
\sigma:\, G_f/U \times \mathfrak{H}_F &\to G_f/U \times \mathfrak{H}_F\\
(x, z) &\mapsto (x^\sigma, z^\sigma).
\end{align*}
This induces an action on the space $S_{k}(U)$. It also acts on $\T_{k}(U)$ by 
sending $T_{\gp}$ to $T_{\sigma(\gp)}$, and $S_{\gp}$ to $S_{\sigma(\gp)}$. Via the
surjection
\begin{align*}
G_f \times G_\infty &\to G_f/U \times \mathfrak{H}_F\\
(g_f, g_\infty) &\mapsto(g_f U, g_\infty(\sqrt{-1},\ldots,\sqrt{-1})),
\end{align*}
this induces an action of $\Gal(F/E)$ on the set of automorphic representations of $\GL_2(\A_F)$
of level $U_1(\gN)$ and weight $k$. 

Let $f \in S_{k}(\gN, \chi)^{\new}$ be a newform, and $\sigma \in G$. Then, there is an automorphic representation 
${}^\sigma\!\pi_f = \pi_f \circ \sigma$ of level $\gN$, weight $k$ and character ${}^\sigma\!\chi := \chi \circ \sigma$. 
Let  ${}^\sigma\!f \in S_{k}(\gN, {}^\sigma\!\chi)$ be the newform such that ${}^\sigma\!f$ is a new vector in ${}^\sigma\!\pi_f$. 
By the Strong Multiplicity One Theorem~\cite{miy71}, ${}^\sigma\!f$ is uniquely determined by the relation
$$a_{\gp}({}^\sigma\!f) := a_{\sigma(\gp)}(f),$$
for all primes $\gp$. 
We say that an automorphic representation $\pi$ is a {\it base change} if $\pi \circ \sigma \simeq \pi$ for
all $\sigma \in  G$. We say that $f$ is a {\it base change} if $\pi_f$ is a base change. This is equivalent to
saying that ${}^\sigma\!f = f$. By the Multiplicity One Theorem, this is also equivalent to saying that 
$$a_{\sigma(\gp)}(f) := a_{\gp}(f),$$
for all $\sigma \in G$ and almost all primes $\gp$. We note that, among other things, this will imply that
$\chi\circ\sigma =\chi$ for all $\sigma \in G$. In other words, $\chi$ must factor through the norm map $\N_{F/E}:\, F \to E$. 

Let $\rho:\,\Gal(\Qbar/F) \to \GL_2(\Qbar_\ell)$ be an $\ell$-adic representation. For $\sigma \in G$,
let $\tilde{\sigma} \in \Gal(\Qbar/E)$ be a lift, and set
$$\rho^\sigma(\delta) = \rho(\tilde{\sigma} \delta \tilde{\sigma}^{-1}).$$
Then, $\rho^\sigma$ is well-defined and depends only on $\sigma$. We call it the $\sigma$-conjugate of $\rho$. 

\begin{conj}[Base Change~\cite{lang80}]\label{conj:bc} Let $f$ be a Hilbert newform of $\Gal(F/E)$-invariant level $\gN$, weight $k$ and character $\chi$.
Then, the followings are equivalent:
\begin{enumerate}[(a)]
\item For all $\sigma \in G$, we have ${}^\sigma\!\pi_f \simeq \pi_f$;
\item For all $\sigma \in G$, we have ${}^\sigma\!f = f$;
\item For all prime $\lambda$ in $L_f$ and $\sigma \in G$, we have $\rho_{f, \lambda}^\sigma = \rho_{f,\lambda}$;
\item There exists a Hilbert newform $\hat{f}$ of level $U_1(\gN')$ and weight $k$ over $E$ such that
$$\rho_{\hat{f}, \lambda'}|_{\Gal(\Qbar/F)} = \rho_{f,\lambda}$$
for all primes $\lambda'$ in $L_{\hat{f}}$ and $\lambda$ in $L_f$ above $\lambda'$;
\item There exists a Hilbert newform $\hat{f}$ of level $U_1(\gN')$ and weight $k$ over $E$ such that
$$L(f, s) = \prod_{\eta \in \widehat{G}} L(\hat{f}\otimes (\eta \circ \Art_F), s),$$
where $\Art_F:\, F^\times\backslash\A_F^\times \to G_F^{\rm ab}$ is the Global Artin reciprocity map. 
\end{enumerate}
(The level $\gN'$ depends on the relative discriminant $\mathfrak{D}_{F/E}$ and the level $\gN$, see Remark~\ref{rem:bc-level}.)
\end{conj}

\begin{rem}\rm One can reformulate Statements (a) and (b) of Conjecture~\ref{conj:bc} by saying that 
$\mathrm{Spec}(\T_{k}(\gN,\chi; \Qbar)^{\new})$ is a $G$-set whose fixed points correspond to newforms 
that are base change.
\end{rem}

\begin{rem}\label{rem:non-solvable-bc}\rm 
We recall that Conjecture~\ref{conj:bc} is true when $F/E$ is a cyclic extension (see~\cite{lang80}). In the non-solvable case, there have
been some progress for $F/\Q$ totally real thanks to Hida~\cite{hida09} and Dieulefait~\cite{die12}. 
\end{rem}

\begin{rem}\label{rem:bc-level}\rm The level $U_1(\gN')$ of the form $\hat{f}$ in Conjecture~\ref{conj:bc} can be determined explicitly using
the local-global compatibility conditions in Theorem~\ref{thm:galois-reps-for-hmf}. We note that, for $F/\Q$ cyclic of prime degree, 
there is earlier work of Saito~\cite[Theorem 4.5]{sai79} which gives $U_1(\gN')$. 
\end{rem}

\section{\bf Base change and Hecke orbits}\label{sec:bc-and-hecke-comp}

We keep the notation of Section~\ref{sec:bc}. In particular, $\gN$ is an integral ideal which is $\Gal(F/E)$-invariant,
and $S_k(\gN)$ is the space of cusp forms of level $\gN$, weight $k$ and trivial character.

Let $\mathscr{S}$ be the set of Hecke orbits of the newforms in $S_k(\gN)$. Let $f \in S_k(\gN)$ be a newform, 
$\tau \in \Hom(L_f, \Qbar)$ and $\sigma \in G$. Then, we have that $L_{{}^\sigma\!f} = L_{f}$, and that
\begin{align*}
a_{\gp}({}^\sigma\!(f^\tau)) = a_{\sigma(\gp)}(f^\tau) = \tau(a_{\sigma(\gp)}(f))= \tau(a_{\gp}({}^\sigma\!f))=a_{\gp}(({}^\sigma\!f)^\tau).
\end{align*}
This means that, by setting $\sigma\cdot[f] := [{}^\sigma\!f]$, we get a well-defined action of $G$ on $\mathscr{S}$. We can write
$\mathscr{S}$ as a disjoint union
$$\mathscr{S} = \coprod_{L}\mathscr{S}_L,$$
where $L$ runs over all fields of coefficients and $\mathscr{S}_L := \left\{[f] : L_f = L \right\}$. 
Our goal is to understand the orbits of this action on this union. To this end, we start with the following lemma, which is somewhat
straightforward. 

\begin{lem}\label{lem:non-bc} Assume that Conjecture~\ref{conj:bc} is true. Let $f \in S_{k}(\gN)$ be a newform with field of coefficients $L$
such that the $G$-orbit of $[f]$ is a singleton.
\begin{enumerate}[(a)]
\item There is a group homomorphism
\begin{align*}
\phi:\, G &\longrightarrow \Aut(L)\\
\sigma &\longmapsto \tau
\end{align*}
where, for every $\sigma \in G$, $\phi(\sigma):=\tau$ is the unique element in $\Aut(L)$ such that ${}^{\sigma}\!f = f^\tau$. 
\item If $f$ is not a base change from $E$ then $\phi$ is non-trivial.
\item If $f$ is not a base change from any intermediate field $E'/E$, then the map $\phi$ is injective. 
\end{enumerate}
\end{lem}

\begin{proof} (a) Let $\sigma \in G$. Since $[f]$ is unique in its $G$-orbit, we have 
$[{}^\sigma\!f] = [f]$.  Hence, there exists $\tau \in \Hom(L, \Qbar)$ such that 
${}^{\sigma}\!f = f^\tau$. So, we have $$a_{\sigma(\gp)}(f) = \tau( a_{\gp}(f)),$$
for all primes $\gp$. By the Strong Multiplicity One Theorem~\cite{miy71}, the element $\tau$ is well-defined and uniquely determined.
Furthermore, since $L_{{}^\sigma\!f} = L_f = L$, and $L$ is generated by the Hecke eigenvalues of $f$, we have that $\tau \in \Aut(L)$.
So, setting $\phi(\sigma) := \tau$, we get a well-defined map $\phi:\, G \to \Aut(L)$. 

\medskip
To prove that $\phi$ is a group homomorphism, let $\phi(\sigma_1) = \tau_1$ and $\phi(\sigma_2) = \tau_2$. Then, for all primes $\gp$, we have 
$$a_{(\sigma_1\sigma_2)(\gp)}(f) = a_{\sigma_1(\sigma_2(\gp))}(f)= \tau_1(a_{\sigma_2(\gp)}(f))=\tau_1\left(\tau_2(a_{\gp}(f))\right)=(\tau_1\tau_2)(a_{\gp}(f)).$$
So we have ${}^{\sigma_1\sigma_2}\!f = f^{\tau_1\tau_2}$, and hence $\phi(\sigma_1\sigma_2)= \phi(\sigma_1)\phi(\sigma_2)$ by the Multiplicity One Theorem. 

\medskip
(b) Since $f$ is not a base change from $E$, there exists $\sigma \in G$ such that ${}^\sigma\!f \neq f$ and $[{}^\sigma\!f] = [f]$.
Therefore, $\phi(\sigma) \neq 1$, hence it must be non-trivial. 
 
\medskip
(c) Assume that there were $\sigma \ne 1$ such that $\phi(\sigma)=1$. This would imply that
${}^{\sigma}\!f = f$. So $f$ would be a base change from the subfield $E' = F^{\langle \sigma \rangle}$, which would be a contradiction. 
So $\phi$ must be injective. 
\end{proof}

\begin{rem}\label{rem:gal-grp}\rm Letting $K = L^\Delta$ in Lemma~\ref{lem:non-bc}, where $\Delta := \mathrm{im}(\phi)$, we see that $\Gal(L/K) = \Delta$.
So, in other words, Lemma~\ref{lem:non-bc} (c) implies that we have an injection $\Gal(F/E) \hookrightarrow \Gal(L/K)$. As we will see 
later, in all our examples where $\Gal(F/E)$ is solvable, this injection is in fact an isomorphism. But there is no reason for this to persist
forever as the degree of the field $L$ becomes larger. 
\end{rem}

\begin{thm}\label{thm:non-bc} Assume Conjecture~\ref{conj:bc} is true.
Let $f \in S_{k}(\gN)$ be a newform with field of coefficients $L$, $G' := \Stab_{G}([f])$ and $E' := F^{G'}$, so that 
$\Gal(F/E') = G'$.
\begin{enumerate}[(a)]
\item There exists a group homomorphism
\begin{align*}
\phi:\, G' &\longrightarrow \Aut(L)\\
\sigma &\longmapsto \tau
\end{align*}
where, for every $\sigma \in G'$, $\phi(\sigma):=\tau$ is the unique element in $\Aut(L)$ such that ${}^{\sigma}\!f = f^\tau$. 
\item If $G'$ is non-trivial, and $f$ is not a base change from $E'$, then $\phi$ is non-trivial. 
\item If $G'$ is non-trivial, and $f$ is not a base change from any intermediate field $F'/E'$, then $\phi$ is an injection. 
\end{enumerate}
\end{thm}

\begin{proof} The $G'$-orbit of $[f]$ is a singleton, so we apply Lemma~\ref{lem:non-bc} relative to the extension $F/E'$. 
\end{proof}

\begin{cor}\label{cor:non-bc-cyclic} Let $F/E$ be a cyclic extension, and $f \in S_{k}(\gN)$ be a newform with field of coefficients $L$.
Let  $\Gal(F/E) = \langle \sigma \rangle$, $\Stab_{G}([f]) = \langle \sigma^s \rangle$ and $E'= F^{\langle \sigma^s\rangle}$.
\begin{enumerate}[(a)]
\item There exists a group homomorphism
\begin{align*}
\Gal(F/E') &\longrightarrow \Aut(L)\\
\sigma^s &\longmapsto \tau
\end{align*}
where $\tau \in \Aut(L)$ is the unique element such that ${}^{\sigma^s}\!f = f^\tau$. 
\item  If $G'$ is non-trivial, and $f$ is not a base change from $E'$, then $\phi$ is non-trivial.
\item If $G'$ is non-trivial, and $f$ is not a base change from any intermediate field $F'/E'$, then $\phi$ is an injection.
\end{enumerate}
\end{cor}

\begin{proof} Since $F/E$ is cyclic, Conjecture~\ref{conj:bc} is true in that case. 
\end{proof}

There is a wide range of combinatorial results that one could derive from Theorem~\ref{thm:non-bc}. 
We only state the following corollaries as an illustration. 

\begin{cor}\label{cor:const-count1} Assume Conjecture~\ref{conj:bc} is true. Let $f \in S_{k}(\gN)$ be a newform whose Hecke constituent is
the unique constituent of dimension $d_f := [L_f : \Q]$. If $d_f < |G|$ then there exists an intermediate field
$E'/E$ such that $f$ is a base change from $E'$. 
\end{cor}

\begin{proof} By assumption, we have $\Stab_{G}([f]) = G$. Assume that there is no intermediate
field $E'$ such that $f$ is a base change from $E'$. Then, the map $\phi$ is injective by Theorem~\ref{thm:non-bc}. 
This implies that $d_f  \ge |G|$, which is a contradiction.
\end{proof}

\begin{cor}\label{cor:const-count2} Assume Conjecture~\ref{conj:bc} is true. Let $f \in S_{k}(\gN)$ be a newform with field of coefficients $L_f$.
Let $s_f := |\Stab_{G}([f])|$ and $n_f$ be the number of Hecke constituents of dimension $d_f := [L_f:\Q]$.
Then, we have $n_f \ge |G|/s_f$. In particular, when $\Stab_{G}([f]) =\{1\}$, there are at least $|G|$
Hecke constituents of dimension $d_f$ (including the ones in the Hecke orbit of $f$). 
\end{cor}

\begin{cor}\label{cor:const-count3}
Assume Conjecture~\ref{conj:bc} is true. Let $f \in S_{k}(\gN)$ be a newform with field of coefficients $L_f$.
Let  $n_f$ be the number of Hecke constituents of dimension $d_f := [L_f:\Q]$.
Assume that $d_f$ is coprime with $|G|$. Then, either $f$ is a base change from $E$ and $n_f =1$, or $f$ is not a base
change from any proper subfield of $E'/E$ and $n_f = |G|$. 
\end{cor}

\begin{rem}\rm In practice, it can be very hard to decide whether a newform  $f$ is a base change or not. 
The thrust of Corollaries~\ref{cor:const-count1}, \ref{cor:const-count2} and~\ref{cor:const-count3} is that they 
allow us to do so by using purely combinatorial arguments sometimes. 
\end{rem}

\begin{rem}\rm An immediate consequence of Corollaries~\ref{cor:const-count2} and~\ref{cor:const-count3}  is that a naive generalisation of
the Maeda conjecture~\cite[Conjecture 1.2]{hm97} to totally real number fields will not work. Any 
proper generalisation must be sensitive to the the action of $\Gal(F/E)$ on Hecke orbits of newforms. 
\end{rem}

\section{\bf Abelian varieties potentially of $\GL_2$-type}\label{sec:potential-gl2-av}

Theorem~\ref{thm:non-bc} and its corollaries have natural implications for the theory of descent of abelian varieties.
Indeed, let $f$ be a newform of weight $2$ and level $\gN$. Assume that $f$ is not a base change from $E$, and
that the map $\phi: \Gal(F/E)\to \Aut(L)$ is an injection. Let $K := L^\Delta$ and $g := [L : K]$, where $\Delta = \mathrm{im}(\phi)$. 
Further assume that there exists an abelian variety $A_f$ which satisfies the Eichler-Shimura construction for $f$ and recall that, 
we have 
$$a_{\sigma(\gp)}(f) = \tau(a_{\gp}(f)),\,\,\text{for all primes}\,\, \gp\,\,\text{and}\,\,\sigma \in \Gal(F/E),$$ 
where $\tau = \phi(\sigma)$. Consider the family of $\lambda$-adic representations
$$\rho_{f, \lambda}:\, \Gal(\Qbar/F) \to \GL_2(L_{f, \lambda})$$
attached to $A_f$, where $\lambda$ runs over all primes in $L_f$. 
The identity above implies that $A_f$ is isogenous to all its Galois conjugates, and that the induced representations
$$\Ind_{E}^{F}\left(\rho_{f, \lambda}\right):\, \Gal(\Qbar/E) \to \GSp_{2g}(L_{f, \lambda}),$$
are irreducible, and are in fact defined over $K$. This suggests that the isogeny class of the abelian variety $A_f$
descends to that of an abelian variety $B/E$ such that $\End_E(B)\otimes \Q \simeq K$. We note that $A_f$
cannot be an $E$-variety, in the terminology of~\cite{ribet04, pyle04} or~\cite{guit10}, since $B$ is not of
$\GL_2$-type. So that motivates the following definition.

\begin{defn} Let $A/E$ be an abelian variety. We say that $A$ is {\rm potentially of $\GL_2$-type} if there 
exists an extension $F/E$ such that $A\times_E F$ is of $\GL_2$-type. 
\end{defn}

We see that an abelian variety that is of $\GL_2$-type is clearly potentially of $\GL_2$-type. We also see 
that being potentially of $\GL_2$-type is slightly stronger than simply acquiring extra endomorphism 
after base change. 
 
Let $A$ be an abelian variety defined over $F$, and $[A]$ its isogeny class. For any $\sigma \in G_E$, set 
$$\sigma\cdot [A] := [A^\sigma],$$
where $A^\sigma$ is the Galois conjugate of $A$ by $\sigma$. This defines an action of $G_E$ on isogeny classes of abelian 
varieties defined over $F$ since $A \sim A'$ implies that $A^\sigma \sim {A'}^\sigma$. We note that, since $A$ is defined over $F$,
this action factor through $G = \Gal(F/E)$. If $A$ is of $\GL_2$-type, then $A^\sigma$ is also of $\GL_2$-type and 
$\End_F(A)\otimes \Q = \End_F(A^\sigma)\otimes \Q$. Let $\Stab_G([A])$ be the stabiliser of the isogeny class of $A$. 
Then, $\Stab_G([A])$ is trivial if and only if $A$ is not isogenous to any of its Galois conjugate. Similarly $\Stab_G([A]) = G$ means that $A$ is 
isogenous to all its Galois conjugates. 

We recall the following well-known lemma.
\begin{lem}\label{lem:galois-action}
Let $A/F$ be an abelian variety of $\GL_2$-type, and $L = \End_F(A)\otimes \Q$. Assume that $A$ is isogenous to all its Galois conjugates. Then, there exists
a group homomorphism $\phi:\,G \to \Aut(L)$.
\end{lem}

\begin{proof} Let $\{\mu_\sigma:\, A^\sigma \to A\}_{\sigma \in G_E}$ be a system of isogenies. For each $\sigma \in G_E$, we define
$$\tau_\sigma(\alpha) := \mu_\sigma \circ \alpha^\sigma \circ \mu_\sigma^{-1},\,\,\alpha \in L.$$
Let $\{\mu_\sigma':\, A^\sigma \to A\}_{\sigma \in G_E}$ be another system of isogenies. Then, we have
$$\mu_\sigma' \circ \alpha^\sigma \circ {\mu_\sigma'}^{-1} = ( \mu_\sigma'\circ \mu_\sigma^{-1}) \circ (\mu_\sigma \circ \alpha^\sigma \circ \mu_\sigma^{-1}) \circ 
(\mu_\sigma \circ {\mu_\sigma'}^{-1}) = \mu_\sigma \circ \alpha^\sigma \circ \mu_\sigma^{-1}$$
since $L$ is commutative. Therefore, $\tau_\sigma$ is well-defined and independent of the choice of the system of isogenies. It is not hard to see that 
$(\phi:\,G_E \to \Aut(L),\,\,\sigma \mapsto \tau_\sigma)$ defines a group action, hence is a homomorphism. Since $A$ is defined over $F$, which is Galois over $E$, 
the map $\phi$ factors through $G$. This concludes the lemma. 
\end{proof}
We recall that an abelian variety $A/F$ of $\GL_2$-type, which is isogenous to all its $\Gal(F/E)$-conjugates, is called an 
{\it $E$-variety} if the homomorphism $\phi$ in Lemma~\ref{lem:galois-action} is trivial. In this case, Ribet~\cite[Theorem 1.2]{ribet94} shows that 
there exists a $2$-extension $\widetilde{F}$ of $F$, and a system of isogenies $\{\mu_\sigma:\, A^\sigma \to A\}_{\sigma \in G}$ defined over $\widetilde{F}$. 

\begin{prop}\label{prop:arise-from} Let $F/E$ be a Galois extension, and $A$ a non ${\rm CM}$ abelian variety of $\GL_2$-type defined over $F$.
Assume that $A$ is an $E$-variety with a system of isogenies $\{\mu_\sigma:\, A^\sigma \to A\}_{\sigma \in G}$ defined over $F$. 
Then, there is an abelian variety $B$ of $\GL_2$-type defined over $E$ such that $A$ is a simple factor of $B\times_{E}F$. 
\end{prop}

\begin{proof} This is an adaptation of the proof of \cite[Proposition 4.5]{pyle04} to arbitrary number fields (see also Ribet~\cite[Theorem 6.1]{ribet04}).
\end{proof}

Let $A/F$ be of $\GL_2$-type, and set $L = \End_F(A)\otimes \Q$. Let $\lambda$ be a prime in $L$ and consider the Galois representation on the
$\lambda$-adic Tate module of $A$
$$\rho_{A,\lambda}:\,\Gal(\Qbar/F) \to \GL_2(L_{\lambda}).$$
Let $\gp$ be a prime in $F$, and set 
$$a_{\gp} = \Tr(\rho_{A, \lambda}(\Frob_\gp)) \in L.$$
Ribet~\cite[Proposition 3.3]{ribet04} shows that $L$ is generated by the $a_\gp$. The content of the following result is that Galois action on isogeny 
classes of abelian varieties of $\GL_2$-type should mirror that on Hecke orbits of Hilbert newforms. 

\begin{thm}\label{thm:potential-gl2} Let $F/E$ be a Galois extension, and $A$ a non {\rm CM} abelian variety of $\GL_2$-type defined over $F$ with
$\End_F(A)\otimes \Q = L$. Assume that $A$ is isogenous to all its Galois conjugates. 
\begin{enumerate}[(a)]
\item There exists a group homomorphism 
$\phi: G \to \Aut(L)$ such that $$a_{\sigma(\gp)} = \tau(a_\gp)$$ for all primes $\gp$ and all $\sigma \in G$, where $\tau:=\phi(\sigma)$. 
\item  If $A$ is not an $E$-variety, then $\phi$ is non-trivial.
\item If $A$ is not an $E'$-variety for any proper subfield $E'/E$, then $\phi$ is an injection. 
\end{enumerate}
\end{thm}

\begin{proof} (a) By Lemma~\ref{lem:galois-action}, there is a homomorphism $\phi:\,G \to \Aut(L)$. We only need to show that
$\phi$ is compatible with the action of $\Gal(F/E)$. Let $\ell$ be a rational prime, and $V_\ell(A) = T_\ell(A)\otimes \Q_\ell$ the 
$\ell$-adic Tate module attached to $A$. This is a $2$-dimensional $L \otimes \Q_\ell$ vector space, and we let 
$$\rho_{A,\ell}:\, \Gal(\Qbar/F) \to \GL(V_\ell(A))$$ be the corresponding Galois representation. For $\sigma \in G_E$, the isogeny
$\mu_\sigma$ induces a $\Q_\ell[\Gal(\Qbar/F)]$-module isomorphism $\mu_\sigma:\, V_\ell(A^\sigma) \stackrel{\simeq}{\to} V_\ell(A)$. 
For all $\alpha \in L$, and $x \in V_\ell(A^\sigma)$, we have
$$\mu_\sigma(\alpha^\sigma x) = (\mu_\sigma \circ \alpha^\sigma \circ \mu_\sigma^{-1})(\mu_\sigma(x)) = \tau(\alpha) \mu_\sigma(x),$$
where $\tau = \phi(\sigma)$. This means that $\mu_\sigma$ becomes an $L$-linear isomorphism by letting $L$ acts on $V_\ell(A^\sigma)$
(resp. $V_\ell(A)$) via $\sigma$ (resp. $\tau$.) Similarly, by definition, there is an isomorphism $V_\ell(A) \stackrel{\simeq}{\to} V_\ell(A^\sigma)$ 
induced by the map $A(\overline{F}) \to A^\sigma(\overline{F})$ sending $x$ to $x^\sigma$, which is $L$-linear if we let $L$ acts on $V_\ell(A^\sigma)$ 
via $\sigma$. From this, we get the diagram below, which is compatible with the action of $\Gal(\Qbar/F)$ as indicated. 
\begin{eqnarray*}
\begin{tikzcd}
V_\ell(A)\arrow[r, "\simeq"]\arrow[loop below, "\alpha"]\arrow[loop above, "\Frob_\gp"] & 
V_\ell(A^\sigma) \arrow[r, "\mu_\sigma"', "\simeq"] \arrow[loop above, "\Frob_{\sigma(\gp)}"] \arrow[loop below, "\alpha^\sigma"]& 
V_\ell(A)\arrow[loop above, "\Frob_{\sigma(\gp)}"] \arrow[loop below, "\tau(\alpha)"]
\end{tikzcd}
\end{eqnarray*}
This implies that
$$\Tr(\rho_{A, \ell}(\Frob_{\sigma(\gp)})) = \tau(\Tr(\rho_{A, \ell}(\Frob_\gp))),$$
which is the stated identity. 
 
(b) By definition $A$ is an $E$-variety if and only if $\phi$ is trivial. 

(c) For every proper subfield $E'/E$, $A$ is not an $E'$-variety. So, $\phi|_{\Gal(F/E')}$ is non-trivial. Therefore, $\phi$ must be injection. 
\end{proof}

The following result is a generalisation of~\cite[Theorem 4.2]{cd17} for cyclic extensions $F/E$. 

\begin{thm}\label{thm:galois-descent}
Let $F/E$ be a Galois extension, and $A$ a non {\rm CM} abelian variety of $\GL_2$-type defined over $F$ such that $\End_F(A)\otimes \Q = L$ is totally real. 
Assume that is $A$ is isogenous to all its Galois conjugates, and that $A$ is not an $E'$-variety for any proper subfield $E'/E$ of $F$.
Then, there exists an abelian variety $B$ defined over $E$, potentially of $\GL_2$-type, such that 
$\End_E(B)\otimes \Q = K$ and  $A \sim B \times_E F$, where $K = L^\Delta$, $\Delta =\mathrm{im}(\phi)$. 
\end{thm}

\begin{proof} A careful inspection shows that the proof of \cite[Theorem 4.2]{cd17} only uses the fact that $L$ is totally real,
and nothing about the solvability of the extension $F/E$.
\end{proof}

\begin{prop}\label{thm:galois-descent} Assume Conjecture~\ref{conj:bc} is true.
Let $f \in S_{2}(\gN)$ be a newform with a totally real field of coefficients $L$ such that $\Stab_{G}([f]) = G$, and
the homomorphism $\phi:\, G \to \Aut(L)$ in Theorem~\ref{thm:non-bc} is injective. 
Assume that the Eichler-Shimura conjecture for totally real fields is true for $f$, i.e., there exists an abelian variety $A/F$,
with $\End_F(A)\otimes \Q = L$, such that 
$$L(A, s) = \prod_{f' \in [f]}L(f',s).$$
Then, there is an abelian variety $B/E$, potentially of $\GL_2$-type, such that $\End_E(B)\otimes\Q = K$ and $A \sim B \times_E F$,
where $K = L^\Delta$, $\Delta = \mathrm{im}(\phi)$. 
\end{prop}

\begin{proof} By construction, the abelian variety $A$ satisfies the condition of Theorem~\ref{thm:potential-gl2}. So, it descends to
an abelian variety $B/E$ potentially of $\GL_2$-type. 
\end{proof}

\begin{rem}\rm By the Gross-Langlands conjecture on the modularity of abelian varieties (see~\cite{gross15} and also~\cite[Conjecture 5.2]{cd17}),
there exists a globally generic cuspidal automorphic representation $\pi$ on $\GSpin_{2g+1}(\A_E)$ such that 
$$L(B,s) = \prod_{\pi' \in [\pi]} L(\pi', s),$$
where $[\pi]$ is the Hecke orbit of $\pi$. 
Since $B$ is the Galois descent of $A$, Langlands Functoriality predicts that $\pi$ must be the automorphic descent from $\GL_{2g}$ to $\GSpin_{2g+1}$ 
of the automorphic induction of $\pi_f$ from $F$ to $E$, whose existence also depends on Conjecture~\ref{conj:bc}. In other words, Theorem~\ref{thm:non-bc}
implies that there is a functorial connection between the Eichler-Shimura conjecture for totally real fields and the Gross-Langlands conjecture for
abelian varieties~\cite{gross15}. 
\end{rem}

\section{\bf An example}\label{sec:gross-ex}

The following beautiful example was first suggested by Benedict H. Gross in connection with his conjecture on the existence of non-solvable number fields 
ramified at one prime only, which we proved for $p=2$ in \cite{dem09}. Unfortunately, all the residual Galois representations involved have solvable images. 
Recently, we realised that this example provides better evidence for the conjectures in~\cite{gross15} (see also~\cite{cd17}). Our example also happens to 
be related to a \verb|mathoverflow.net| question, which was partially answered by Elkies~\cite{elk15}.

\subsection{The Shimura curve}
Let $F = \Q(\alpha) = \Q(\zeta_{32}+\zeta_{32}^{-1})$ be the maximal totally real subfield of the cyclotomic field of the $32$nd root of unity. 
This field is defined by the polynomial $x^8 - 8x^6 + 20x^4 - 16x^2 + 2$. Let $\CO_F$ be the ring of integers of $F$. Let 
$v_1,\,\ldots,v_8$ be the real places of $F$. We consider the quaternion algebra $D/F$ ramified at $v_2,\ldots, v_8$ and the
unique prime $\gq$ above $2$. More concretely, we have $D = \left(\frac{u, -1}{F}\right)$, where $u = -\alpha^2 + \alpha$ has 
signature $(+, -,\ldots,-)$. Let $\CO_D$ be a maximal order in $D$, and $X_0^D(1)$ the Shimura curve attached to $\CO_D$. 
Let $w_D$ be the Atkin-Lehner involution at $\gq$. We also let $D'/F$ be the totally definite quaternion algebra ramified exactly 
at all the real places $v_1,\ldots,v_8$, and fix a maximal order $\CO_{D'}$ in $D'$. 

\medskip
Let $S_2(\gq)^{\rm new}$ be the new subspace of cusp forms of level $\gq$ and weight $2$, this is a $40$-dimensional space. 
Let $S_2^D(1)$ (resp. $S_2^{D'}(\gq)^{\rm new}$) be the space of cusp forms of level $(1)$ on $D$ (resp. new subspace
of cusp forms of level $\gq$ on $D'$.) By the Jacquet-Langlands correspondence, we have isomorphisms of Hecke modules
$$S_2(\gq)^{\rm new} \simeq S_2^D(1) \simeq S_2^{D'}(\gq)^{\rm new}.$$ Moreover, we can canonical identify $S_2^D(1)$
with the the space of $1$-differential forms on $X_0^D(1)$. The space $S_2(\gq)^{\rm new}$ decomposes into $5$ Hecke constituents of dimensions 
$4, 4, 4, 4$ and $24$ respectively. (We note that all the computations have been performed using the Hilbert Modular Forms Package in 
\verb|Magma|~\cite{magma}, see also~\cite{dd08, dv13, gv11}.) We let $f, f', g, g'$ and $h$ be newforms in those constituents. Then we have:

\begin{table}
\caption{Eigenforms of level $\gq$ and weight $2$ on $F = \Q(\zeta_{32})^{+}$}
\begin{tabular}{  >{$}c<{$}  >{$}c<{$}  >{$}c<{$}  >{$}c<{$} }\toprule
\text{Newform}&\text{Coefficient field}\,\, L_f&\text{Fixed field}\,\, K_f = L_f^\Delta & \Gal(L_f/K_f) \\\midrule
f, f' & \Q(\zeta_{15})^+& \Q & \Z/4\Z\\
g, g' & \text{Quartic subfield of}\, \Q(\zeta_{95})^+& \Q & \Z/4\Z\\
\multirow{2}{*}{$h$}  &\text{Ray class field of modulus} &  \Q(c) := \Q[x]/(r(x)),& \multirow{2}{*}{$\Z/8\Z$}\\
& \gc = (\frac{1}{2}(c^2 - 16c + 25)) &  r = x^3 + x^2 - 229x + 167& \\
\bottomrule
\end{tabular}
\label{table:eigenforms}
\end{table}

\begin{enumerate}[(i)]
\item The forms $f$ and $f'$ have the same coefficient field, which is the real quartic field $\Q(\zeta_{15})^+$ 
given by $x^4 + x^3 - 4x^2 - 4x + 1$; 
\item  The forms $g$ and $g'$ have the same coefficient field, which is the real quartic subfield of $\Q(\zeta_{95})^+$
given by $x^4 + 19x^3 - 59x^2 + 19x + 1$;
\item The coefficient field of the form $h$ is a field $L_h$ of degree $24$, which is cyclic over the field $K_h = \Q(c)$ defined by 
$c^3 + c^2 - 229c + 167 = 0$. More precisely, it is the ray class field of conductor $\gc = (\frac{1}{2}(c^2 - 16c + 25))$. 
\end{enumerate} 
(We summarise that data in Table~\ref{table:eigenforms}, and the relations among the forms in Table~\ref{table:eigen-identities}.)
Let $w$ and $w_D$ be the Atkin-Lehner involutions acting on $S_2(\gq)^{\rm new}$ and $S_2^D(1)$ respectively.
The Atkin-Lehner involution $w$ acts as follows:
\begin{align*}
w f =  -f,\,\,w f' = -f',\,\,w g  = -g,\,\,w g'  = -g',\,\,w h  = h.
\end{align*}
We recall that $w_D = -w$. 

From the above discussion, it follows that $X_0^D(1)$ is a curve of genus $40$. 

\begin{lem}\label{lem:field-of-defn} 
The curve $X_0^D(1)$ and the Atkin-Lehner involution $w_D$ are both defined over $\Q$.
\end{lem}

\begin{proof}
Since $\sigma(\gq) = \gq$ and the ray class group of modulus $\gq v_2\cdots v_8$ is trivial, the curve $X_0^D(1)$ is defined over $F$ by~\cite[Corollary]{doi-naga67},
and the field of moduli is $\Q$. Furthermore, the field $\Q(\zeta_{32})$ is a splitting field for $D$ whose class number is one. So, the CM point attached to the extension $\Q(\zeta_{32})/F$ is defined over $F$. Therefore, by~\cite[Corollary 1.11]{sijsling-voight15} the curve $X_0^D(1)$ descends to $\Q$. 

\medskip
Alternatively, by using the moduli interpretation in~\cite{car86b}, or the more recent work~\cite{tixi16}, 
one can show that both $X_0^D(1)$ and $w_D$ are defined over $\Q$. 
\end{proof}

\begin{cor}\label{cor:quotient-curve} The curve $C := X_0^D(1)/\langle w_D \rangle$ has genus $16$,
and descends to $\Q$.
\end{cor}

\subsection{The Jacobian varieties $\Jac(X_0^D(1))$ and $\Jac(C)$}
From the above discussion, we have the following decomposition for $\mathrm{Jac}(X_0^D(1))$ over $F$:
\begin{align}\label{eq:jac-decomp}
\mathrm{Jac}(X_0^D(1)) = A_f \times A_{f'} \times A_g \times A_{g'}\times A_h.
\end{align}
From~\eqref{eq:jac-decomp}, and the fact that $w_D = -w$, we see that 
\begin{align*}
\Jac(C) = A_f \times A_{f'} \times A_g \times A_{g'}.
\end{align*}
The fourfolds $A_f$ and $A_{f'}$ (resp. $A_{g}$ and $A_{g'}$) are Galois conjugate. 
We will see later that one of consequences of the compatibility between the base change action and 
Hecke orbits is that the decomposition~\eqref{eq:jac-decomp} descends to subfields of $F$.

\begin{thm}\label{thm:Ah} The abelian variety $A_h$ descends to a $24$-dimensional variety $B_h$ defined over $\Q$,
with good reduction outside $2$, such that $\End_{\Q}(B_h)\otimes\Q = K_h$ and 
$$L(B_h, s) = \prod_{\Pi' \in [\Pi_h]} L(\Pi', s),$$
where $\Pi_h$ is the automorphic representation lifting $\pi_h$ to $\GSpin_{17}(\A_{\Q})$ and $[\Pi_h]$ its Hecke orbit. 
\end{thm}

\begin{proof}
By Table~\ref{table:eigen-identities}, there exists a generator $\tau \in \Gal(L_h/K_h)$ such that  ${}^\sigma\! h = h^\tau$. 
So, by ~\cite[Theorem 5.4]{cd17}, $\pi_h$ lifts to an automorphic representation $\Pi_h$ on a 
split form of $\GSpin_{17}(\A_{\Q})$ with coefficients in the cubic field $K_h$. The Hecke orbit $[\Pi_h]$ of $\Pi_h$ 
has $3$ elements, and by functoriality 
$$L(B_h, s) = \prod_{\Pi' \in [\Pi_h]} L(\Pi', s).$$
It follows that $\End_{\Q}(B_h)\otimes\Q = K_h$. Since the level of the form $h$ is the unique prime $\gq$ above $2$,  
$B_h$ has good reduction outside $2$. 
\end{proof}

\medskip
Now, we turn to the quotient $C := X_0^D(1)/\langle w_D \rangle$.

\begin{thm}\label{thm:fourfolds}
The abelian varieties $A_f$ and $A_{f'}$ (resp. $A_{g}$ and $A_{g'}$) descend to pairwise conjugate fourfolds $B_f$ and 
$B_{f'}$ (resp. $B_{g}$ and $B_{g'}$) over $\Q(\sqrt{2})$ with trivial endomorphism rings such that 
\begin{align*}
L(B_f, s) &= L(\Pi_f, s)\,\,\text{and}\,\, L(B_{f'}, s) = L(\Pi_{f'}, s),\\
L(B_g, s) &= L(\Pi_g, s)\,\,\text{and}\,\, L(B_{g'}, s) = L(\Pi_{g'}, s),
\end{align*}
where $\pi_f, \pi_{f'}, \pi_g$ and $\pi_{g'}$ lift to the automorphic representations $\Pi_f, \Pi_{f'}, \Pi_g$ and $\Pi_{g'}$ 
on $\GSpin_{9}/\Q(\sqrt{2})$. They have good reduction outside $(\sqrt{2})$. 
\end{thm}

\begin{proof} The Identity II in Table~\ref{table:eigen-identities}, combined with \cite[Theorem 5.4]{cd17}, implies that 
$\pi_f, \pi_{f'}, \pi_g$ and $\pi_{g'}$ lift to automorphic representations $\Pi_f, \Pi_{f'}, \Pi_g$ and $\Pi_{g'}$ on 
$\GSpin_9/\Q(\sqrt{2})$ (or $\SO(9)$ after normalisation) with coefficients in $\Q$. 
Consequently, the fourfolds $A_f$, $A_{f'}$, $A_g$ and $A_{g'}$ descend to pairwise conjugate fourfolds $B_f$ and $B_{f'}$ (resp. $B_g$ and $B_{g'}$) 
such that
$$\End_{\Q(\sqrt{2})}(B_f) = \End_{\Q(\sqrt{2})}(B_{f'}) = \End_{\Q(\sqrt{2})}(B_g) = \End_{\Q(\sqrt{2})}(B_{g'}) = \Z.$$
The equalities of $L$-series follow by functoriality. For the same reason as above, the fourfolds have good reduction outside $(\sqrt{2})$. 
\end{proof}

\begin{table}
\caption{Identities among eigenforms in Table~\ref{table:eigenforms}. (Here, we have $\Gal(F/\Q) = \langle \sigma \rangle$, with
$\sigma:\, \alpha \mapsto -\alpha^5 + 5\alpha^3 - 5\alpha$.)}
\begin{tabular}{  >{$}c<{$}  >{$}c<{$}  >{$}c<{$} >{$}c<{$} }\toprule
\text{Newform} & \tau & \text{I} & \text{II} \\\midrule
f, f'  & b \mapsto  -b^3 + b^2 + 3b - 2& {}^\sigma\!f = f' & {}^{\sigma^2}\!f = f^\tau \\ 
g, g' &  b \mapsto \frac{1}{11}(-3b^3 - 58b^2 + 154b - 35)& {}^\sigma\!g = g' & {}^{\sigma^2}\!g = g^\tau \\ 
h & \exists\tau & &{}^\sigma\!h = h^\tau\\
\bottomrule 
\end{tabular}
\label{table:eigen-identities}
\end{table}

\begin{rem}\rm The decomposition~\eqref{eq:jac-decomp} is only true {\it a priori} over $F$. 
However, Theorem~\ref{thm:Ah} and Theorem~\ref{thm:fourfolds} imply that it descends to $\Q(\sqrt{2})$. In fact, it
will further descend to $\Q$ if we put the pairwise conjugates $A_f$ and $A_{f'}$ (resp. $A_g$ and $A_{g'}$) together.
\end{rem}

\subsection{The residual Hecke algebras}
Let $\mathbf{T}^{\rm new}$ be the $\Z$-subalgebra of $\End_\C(S_{2}(\gq))$ acting on $S_2(\gq)^{\rm new}$; and $\mathbf{T}_f$, $\mathbf{T}_{f'}$ $\mathbf{T}_g$ $\mathbf{T}_{g'}$
and $\mathbf{T}_h$ the $\Z$-subalgebras acting on the constituents of $f$, $f'$, $g$, $g'$ and $h$ respectively. From the above discussion, we have 
\begin{align*}
\mathbf{T}^{\rm new}\otimes\Q &= (\mathbf{T}_f \otimes \Q) \times (\mathbf{T}_{f'} \otimes \Q) \times (\mathbf{T}_g \otimes \Q) \times (\mathbf{T}_{g'} \otimes \Q)
 \times (\mathbf{T}_h \otimes \Q)\\
&= L_f \times L_{f'} \times L_{g} \times L_{g'} \times L_{h}.
\end{align*}
By direct calculations, we get the followings:
\begin{itemize}
\item $[\CO_{L_f} : \mathbf{T}_f] = [\CO_{L_{f'}} : \mathbf{T}_{f'}]$ divides $3$,
\item $[\CO_{L_g} : \mathbf{T}_g]  = [\CO_{L_{g'}} : \mathbf{T}_{g'}] = 1$,
\item $[\CO_{L_h} : \mathbf{T}_h]$ divides $3\cdot 5^6$.
\end{itemize}
Therefore $\mathbf{T}^{\rm new} \otimes \Z_2$ decomposes into $\Z_2$-algebras as
$$\mathbf{T}^{\rm new} \otimes \Z_2 = (\mathbf{T}_f \otimes \Z_2) \times (\mathbf{T}_{f'} \otimes \Z_2) \times (\mathbf{T}_g \otimes \Z_2) \times (\mathbf{T}_{g'} \otimes \Z_2)
\times (\mathbf{T}_h \otimes \Z_2).$$
Since the prime $2$ is inert in $L_f = L_{f'}$, and $L_g = L_{g'}$, the first four factors of this decomposition are local $\Z_2$-algebras. 
Let $\gm_f$, $\gm_{f'}$, $\gm_g$ and $\gm_{g'}$ be the corresponding maximal ideals. Then, by the identities in Table~\ref{table:eigen-identities}, 
we have $\sigma(\gm_f) = \gm_{f'}$ and $\sigma^2(\gm_f) = \tau_f(\gm_f)$; and  $\sigma(\gm_g) = \gm_{g'}$ and $\sigma^2(\gm_g) = \tau_g(\gm_g)$. 
We let $\theta_f:\,\mathbf{T}_f\otimes \Z_2 \to \F_{16}$, $\theta_{f'}:\,\mathbf{T}_{f'}\otimes \Z_2 \to \F_{16}$, 
$\theta_g:\,\mathbf{T}_g\otimes \Z_2 \to \F_{16}$ and $\theta_{g'}:\,\mathbf{T}_{g'}\otimes \Z_2 \to \F_{16}$ be the corresponding mod $2$ Hecke eigensystems.

Next, we recall that $L_h$ is the ray class field of conductor $\gc = (\frac{1}{2}(c^2 - 16c + 25))$ over the field $K_h = \Q(c)$, with $c^3 + c^2 - 229c + 167 = 0.$
The prime $2$ is totally ramified in $K_h$. Letting $\gp_2$ be the unique prime above it, we get that $\gp_2 = \gP \gP'$, where $\gP$ and $\gP'$ are inert primes,
and $\tau(\gP) = \gP'$. Therefore, there are two maximal primes $\gm_{h}$ and $\gm_{h}'$ in $\mathbf{T}_h \otimes \Z_2$ such that 
$\sigma(\gm_{h}) = \gm_{h}'$ and $\sigma^2(\gm_{h}) = \tau_h(\gm_{h})$. We let $\theta_h \times \theta_h':\,\mathbf{T}_h\otimes \Z_2 \to \F_{16} \times \F_{16}$ be
the resulting two mod $2$ Hecke eigensystems. 
 
By computing the socle of the underlying $\F_2$-module to $\mathbf{T}^{\rm new} \otimes \F_2$, we obtain that
$\theta_f \simeq \theta_g \simeq \theta_h$, and $\theta_{f'} \simeq \theta_{g'} \simeq \theta_{h}'$, up to rearranging. 
We will denote these two Hecke eigensystems by $\theta$ and $\theta'$ respectively.

\subsection{The field of $2$-torsion for $\Jac(X_0^D(1))$ and $\Jac(C)$} 
To analyse the field of $2$-torsion for our varieties, we start with the following proposition. We recall the following diagram
\begin{eqnarray*}
\begin{tikzcd}
&&\Q(\zeta_{64})&&\\
\Q(\zeta_{64})^{+}\arrow[-, "2"]{urr}&&\Q(i(\zeta_{64} + \zeta_{64}^{-1}))\arrow[-, "2"']{u}&&\Q(\zeta_{32})\arrow[-, "2"']{ull}\\
&&F\arrow[-, "2"']{ull}\arrow[-, "2"']{u}\arrow[-, "2"]{urr}&&\\
\end{tikzcd}
\end{eqnarray*}
The subfield $K = \Q(\beta) = \Q(i(\zeta_{64} + \zeta_{64}^{-1}))$ is the unique CM extension of $F$ with class number $17$. 
For later, we observe that $\beta^2 = -2 - \alpha$, where $\gq = (2 + \alpha)$.

\begin{prop}\label{prop:dihedral-rep} 
Let $\bar{\rho},\, \bar{\rho}': \Gal(\Qbar/F) \to \GL_2(\F_{16})$ be the mod $2$ Galois representations attached to $\theta$ and $\theta'$ respectively.
Then, there are characters $\chi, \chi': \Gal(\Qbar/K) \to \F_{2^8}^\times$, with trivial conductor such that $\bar{\rho} = \Ind_{K}^F \chi$, and 
$\bar{\rho}' = \Ind_{K}^F \chi'$. 
\end{prop}

\begin{proof} We already computed the Hecke constituents of the space $S_2(1)$ in~\cite{dem09}. The mod $2$ Hecke eigensystems in that case
have coefficient fields $\F_{2^s}$ where $s = 1, 2, 8$. Therefore, since $\theta$ has coefficient field $\F_{16}$, it cannot arise from an eigenform of level $1$. 
By the Serre conjecture for totally real fields (the totally ramified case)~\cite{gs11}, it must appear on the quaternion algebra $D'$ with level $(1)$ and non-trivial 
weight. The same is true for $\theta'$. In fact, the analysis conducted above shows that they are the only eigensystems that can appear at that weight. 
(We note that there are only two Serre weights in this case.)

\medskip
Let $\chi:\, \Gal(\Qbar/K) \to \overline{\F}_2^\times$ be a character with trivial conductor such that $\chi^s \neq \chi$, where $\Gal(K/F) = \langle s \rangle$. 
By class field theory, we can identity $\chi$ with its image under the Artin map. Since $\chi$ is unramified, it must factor as 
$\chi: K^\times \backslash \A_K^\times \twoheadrightarrow \Cl_K \to \overline{\F}_2^\times$.
Furthermore, since $\Cl_K \simeq \Z/17\Z$, we must have 
$\chi: K^\times \backslash \A_K^\times \to \F_{2^8}^\times$, and the representation $\bar{\rho}_{\chi}:=\Ind_K^F \chi:\, \Gal(\Qbar/F) \to \GL_2(\F_{16})$ has 
coefficients in $\F_{16}$. So, $\bar{\rho}_{\chi}$ has level $(1)$ and non-trivial weight by the argument above. Therefore, it must be a 
Galois conjugate of $\bar{\rho}$. Up to relabelling, we can assume that 
$\bar{\rho} \simeq \bar{\rho}_{\chi}$. Since $\theta$ and $\theta'$ are $\Gal(F/\Q)$-conjugate, there is also a character 
$\chi': K^\times \backslash \A_K^\times \to \F_{2^8}^\times$ such that $\bar{\rho}' \simeq \bar{\rho}_{\chi'}$. 

\medskip
Alternatively, we can show that $\theta$ appears on $D'$ with the non-trivial weight without
using the fact that it has coefficients in $\F_{16}$. Indeed, we have 
$$\bar{\rho}_{\chi}|_{I_\gq} \simeq \begin{pmatrix} 1& \ast\\ 0 & 1\end{pmatrix}.$$
Let $K_{\gQ}$ be the completion of $K$ at $\gQ$, the unique prime above $\gq$. 
Since $K = F[\beta]$, and $\beta^2 = -2 - \alpha$ is a generator of $\gq$, then we have $K_{\gQ} = F_\gq[\sqrt{\varpi}]$,
where $\varpi$ is a uniformiser of $F_\gq$. Therefore, $\bar{\rho}_{\chi}|_{D_\gq}$ doesn't arise from a finite flat group 
scheme. Hence, $\bar{\rho}_{\chi}$ must have non-trivial weight. 
\end{proof}

We are now ready to state the main theorem of this section. 

\begin{thm}\label{thm:2-torsion} The Jacobians $\Jac(X_0^D(1))$ and $\Jac(C)$ of the curves $X_0^D(1)$ and $C$,
both defined over $\Q$, have the same field of $2$-torsion $N$. The field $N$ is the unique field unramified outside $2$, 
with Galois group $\Gal(N/\Q) = \Z/17\Z \rtimes (\Z/17\Z)^\times = F_{17}$. 
\end{thm}

\begin{proof} Let $\bar{\rho}_{f,2},\,\bar{\rho}_{f',2}:\,\Gal(\Qbar/F) \to \GL_2(\F_{16})$ be the mod $2$ Galois representations attached to $f$ and $f'$. By
Proposition~\ref{prop:dihedral-rep}, $\bar{\rho}_{f,2}$ and $\bar{\rho}_{f',2}$ are dihedral and we have that 
$\mathrm{im}(\bar{\rho}_{f,2}) = \mathrm{im}(\bar{\rho}_{f',2}) = D_{17}$. Let $M_f, M_{f'}$ be the fields cut out by $\bar{\rho}_{f,2}$ and $\bar{\rho}_{f',2}$.
Since $\sigma(\gm_f) = \gm_{f'}$, we have $M_{f'} = M_f^\sigma$.  By construction $M_f$ and $M_f^\sigma$ are unramified extension of $K$. So,
by uniqueness of the Hilbert class field, we must have $M_f = M_f^\sigma = M_f M_f^\sigma= H_K$, where $M_f M_f^\sigma$ is the
compositum of $M_f$ and $M_f^\sigma$; and $H_K$ is the Hilbert class field of $K$. 
Since $\sigma^2(\gm_f) = \tau_f(\gm_f)$, we have 
$$\Gal(N_f/\Q) = D_{17}\rtimes \Z/8\Z = F_{17}.$$
Letting $N_g$ and $N_h$ be the normal closures of the fields cut out by the mod $2$ representations attached to $g$ and $h$ respectively, 
the same argument shows that 
$$\Gal(N_g/\Q) = \Gal(N_h/\Q) = F_{17}.$$
By~\cite[Theorem 2.25]{har94}, there is a unique field $N$ ramified at $2$ and $\infty$, with Galois group $F_{17}$. 
Therefore, we must have $N = N_f = N_g = N_h.$
By using the decomposition in~\eqref{eq:jac-decomp}, we see that $N$ must be the field of $2$-torsion for both $\Jac(X_0^D(1))$ and $\Jac(C)$. 
\end{proof}

\begin{rem}\rm The field $N$ is the splitting field of the polynomial
\begin{align*}
H &:= x^{17} - 2x^{16} + 8x^{13} + 16x^{12} - 16x^{11} + 64x^9 - 32x^8 - 80x^7 + 32x^6 + 40x^5\\
&\qquad{} + 80x^4 + 16x^3 - 128x^2 - 2x + 68. 
\end{align*}
This polynomial was computed by Noam Elkies. His computation, together with our Theorem~\ref{thm:2-torsion}, answer a question
posed by Jeremy Rouse as to whether the field $N$ is the field of $2$-torsion of some abelian variety. We thank David P. Roberts for
bringing this \verb|mathoverflow.net| discussion~\cite{elk15} to our attention. 
\end{rem}

One can show that the field $K$ splits the quaternion algebra $D$. Let $\CO$ be the suborder of $\CO_K$ of index $\gQ^2$, where 
$\gQ$ is the unique prime above $\gq$ in $K$. Let $\Pic(\CO)$ be the Picard group of $\CO$. A quick \verb|Magma| calculation 
show that $\#\Pic(\CO) = 34 = 2\cdot 16 + 2$. We conclude this note with the following claim. 

\begin{conj}\label{conj:hyperelliptic} 
The curve $C$ is hyperelliptic over $F$. Its Weierstrass points are the ${\rm CM}$ points arising from $\Pic(\CO)$. 
\end{conj}

\begin{rem}\rm We believe that the curve $C$ is in fact hyperelliptic over $\Q$.
\end{rem}

\begin{rem}\rm We observe that $w_D$ is the unique Atkin-Lehner involution on $X_0^D(1)$. Therefore, if Conjecture~\ref{conj:hyperelliptic} is true, 
then the hyperelliptic involution on $C$ must be an exceptional one. We note that, for $F = \Q$, Michon~\cite{mich81} provides a complete list of all Shimura 
curves with square-free level that are hyperelliptic. (This result was independently obtained by Ogg in an unpublished work.) But, in general, the question
of finding those Shimura curves which admit a hyperelliptic quotient is still wide open even for $F = \Q$. In that regards, 
Conjecture~\ref{conj:hyperelliptic} is rather striking. 
\end{rem}

\begin{rem}\rm One should be able to find an explicit equation for $C = X_0^D(1)/\langle w_D \rangle$ using~\cite{voight-willis14}.
But, currently, the strategy for doing so is not fully implemented. It should also be possible to use a generalisation of the $p$-adic 
approach discussed in~\cite{cm14}, which was inspired by~\cite{kur94}. We hope to return to this problem in an upcoming paper. 
\end{rem}

\providecommand{\bysame}{\leavevmode\hbox to3em{\hrulefill}\thinspace}
\providecommand{\MR}{\relax\ifhmode\unskip\space\fi MR }
\providecommand{\MRhref}[2]{%
  \href{http://www.ams.org/mathscinet-getitem?mr=#1}{#2}
}
\providecommand{\href}[2]{#2}

\end{document}